\documentclass[a4paper]{article}

\usepackage{amsthm, amssymb, amsmath}
\usepackage{mathrsfs, latexsym, color}
\usepackage[pdftex]{graphicx}
\usepackage{longtable,textcomp}
\usepackage{cite}
\numberwithin{equation}{section}

\theoremstyle{plain}
\newtheorem{theorem}{Theorem}[section]
\newtheorem{proposition}[theorem]{Proposition}
\newtheorem{lemma}[theorem]{Lemma}
\newtheorem{corollary}[theorem]{Corollary}
\newtheorem{remark}{Remark}

\usepackage{fullpage}

\begin{document}
\title{ Improving the Delsarte bound } 

\author{Gary R.W. Greaves$^{a}$, Jack H. Koolen$^{b,c}$ and  Jongyook Park$^{d}$\\
\\
\small $^a$School of Physical and Mathematical Sciences,\\
\small Nanyang Technological University, 21 Nanyang Link, Singapore 637371, Singapore\\
\small $^b$School of Mathematical Sciences,\\
\small University of Science and Technology of China, 96 Jinzhai Road, Hefei,  Anhui, 230026, PR China\\
\small $^c$CAS Wu Wen-Tsun Key Laboratory of Mathematics,\\
\small University of Science and Technology of China, 96 Jinzhai, Road, Hefei, Anhui, 230026, PR China\\
\small $^d$Department of Mathematics, Kyungpook National University, Daegu, 41566, Republic of Korea\\
\small {\tt e-mail: gary@ntu.edu.sg, koolen@ustc.edu.cn, jongyook@knu.ac.kr}\vspace{-3pt}
}

 \maketitle

\begin{abstract}
 In this paper, we study the order of a maximal clique in an amply regular graph with a fixed smallest eigenvalue by considering a vertex that is adjacent to some (but not all) vertices of the maximal clique. As a consequence, we show that if a strongly regular graph contains a Delsarte clique, then the parameter $\mu$ is either small or large. 
 Furthermore, we obtain a cubic polynomial that assures that a maximal clique in an amply regular graph is either small or large (under certain assumptions). 
Combining this cubic polynomial with the claw-bound, we rule out an infinite family of feasible parameters $(v,k,\lambda,\mu)$ for strongly regular graphs. 
Lastly, we provide tables of parameters $(v,k,\lambda,\mu)$ for nonexistent strongly regular graphs with smallest eigenvalue $-4, -5, -6$ or $-7$.


\end{abstract}

\section{Introduction}
The following bounds are well-known for the order of a clique in a graph, regular graph, and strongly regular graph, respectively. 
\begin{itemize}
\item[$(i)$] For a graph $\Gamma$ with $v$ vertices, Cvetkovi\'{c} proved that the order of a coclique in $\Gamma$ is at most ${\rm min}\{v-n_+, v-n_-\}$ \cite[Theorem 3.5.1]{Spectra}, where $n_+$ and $n_-$ are the numbers of positive and negative eigenvalues of $\Gamma$, respectively. We call this bound the {\em Cvetkovi\'{c}  bound}. 

\item[$(ii)$] For a regular graph with valency $k$ and smallest eigenvalue $-m$, Hoffman proved that the order of a coclique in $\Gamma$ is at most $v\frac{m}{k+m}$ \cite[Theorem 3.5.2]{Spectra}. We call this bound  the {\em Hoffman bound}.

\item[$(iii)$] For a strongly regular graph with valency $k$ and smallest eigenvalue $-m$,  Delsarte proved that the order of a clique in $\Gamma$ is at most $1+\frac{k}{m}$ \cite[Section 3.3.2]{Delsarte}. We call this bound the {\em Delsarte bound} and a clique in $\Gamma$ is called a {\em Delsarte clique} if its order is equal to $1+\frac{k}{m}$. 
\end{itemize}

From the Hoffman bound and  Cvetkovi\'{c} bound, we can obtain an upper bound on the order of a clique in a graph $\Gamma$ by considering the complement of $\Gamma$. We note that if the graph $\Gamma$ is a strongly regular graph, then the bound obtained from the Hoffman bound is the same as the Delsarte bound \cite[Proposition 1.3.2]{bcn89}.

Our purpose is to study the order of a maximal clique in an amply regular graph with a fixed smallest eigenvalue by considering a vertex that is adjacent to some (but not all) vertices of the maximal clique. 
As a consequence, we show that if a strongly regular graph contains a Delsarte clique, then the parameter $\mu$ is either small or large.
Our main tool, which we exhibit in Section~\ref{sec:poly} is a cubic polynomial corresponding to an amply regular graph $\Gamma$ that can be used to bound that size of a maximal clique in $\Gamma$.
In Section~\ref{sec:claw}, we combine this cubic polynomial with the claw-bound to rule out an infinite family of feasible parameters $(v,k,\lambda,\mu)$ for strongly regular graphs. 
In the appendix, we provide tables of parameters $(v,k,\lambda,\mu)$ for nonexistent strongly regular graphs with smallest eigenvalue $-4, -5, -6$ or $-7$.

\section{Definitions and preliminaries}
\label{sec:2}
All the graphs considered in this paper are finite, undirected and simple. 
For basic definitions and terminology, the reader is referred to \cite{bcn89}.
Let $\Gamma$ be a connected graph with vertex set $V(\Gamma)$. 
The {\em distance} $d(x,y)$ between two vertices $x,y\in V(\Gamma)$ is the length of a shortest path between $x$ and $y$ in $\Gamma$. 
The maximum distance occuring in $\Gamma$ is called the {\em diameter} of $\Gamma$. 
For each $x\in V(\Gamma)$, denote by $\Gamma(x)$ the set of vertices in $\Gamma$ that are adjacent to $x$.  
For a vertex $x$ of $\Gamma$, the number $|\Gamma(x)|$ is called the valency of $x$ in $\Gamma$. 
In particular, if $k=|\Gamma(x)|$ holds for all $x\in V(\Gamma)$, then $\Gamma$ is {\em regular} with {\em valency} $k$ .

A regular graph with $v$ vertices and valency $k$ is called {\em edge-regular} with parameters $(v,k,\lambda)$ if any two adjacent vertices have exactly $\lambda$ common neighbors. An edge-regular graph with parameters $(v,k,\lambda)$ is called {\em amply regular} with parameters $(v,k,\lambda,\mu)$ if any two vertices at distance $2$ have exactly $\mu$ common neighbors. An amply regular graph with parameters $(v,k,\lambda,\mu)$ with diameter at most $2$ is also called {\em strongly regular} with parameters $(v,k,\lambda,\mu)$.

\begin{theorem}{\rm (Cf.\ \cite[Theorem 8.6.3]{BVM})}\label{SNB}
Let $m\geqslant2$ be an integer. Let $\Gamma$ be a strongly regular graph with parameters $(v,k,\lambda,\mu)$ and eigenvalues $k>\sigma>\tau$. If $\tau=-m$ and $\sigma>\frac{1}{2}m(m-1)(\mu+1)-1$, then one of the following holds: 
\begin{itemize}
\item[$(i)$] $\mu=m(m-1)$ and  $\Gamma$ is a Latin square graph $LS_m(n)$,
\item[$(ii)$] $\mu=m^2$ and $\Gamma$ is a block graph of Steiner system $2$-$(mn+m-n,m,1)$,
\end{itemize}
where $\sigma=n-m$.
\end{theorem}

We use the above theorem as follows.
Let $m\geqslant2$ be an integer and let $\Gamma$ be a strongly regular graph with parameters $(v,k,\lambda,\mu)$ and smallest eigenvalue $-m$.  Assume that $\mu\neq m(m-1)$ and $\mu\neq m^2$, then by Theorem~\ref{SNB}, we have $\sigma\leqslant \frac{1}{2}m(m-1)(\mu+1)-1$.  Since $\lambda-\mu=\sigma-m$ and $k-\mu=\sigma m$ (\cite[p.\ 219]{GodsilRoyle}),  we have $\lambda=\sigma+\mu-m\leqslant\frac{1}{2}(m^2-m+2)(\mu-1)+m^2-2m$ and $k=\sigma m+\mu\leqslant \frac{1}{2}m^2(m-1)(\mu+1)-m+\mu$. Furthermore, one can see that $v=1+k+k(k-\lambda-1)/\mu$. Neumaier~\cite[Theorem 3.1]{Neumaier} showed that $\mu\leqslant m^3(2m-3)$ - we refer to this bound as the {\em $\mu$-bound}. 
The $\mu$-bound shows that there are finitely many such strongly regular graphs (for fixed $m$).  
For a pair $(\lambda,\mu)$ satisfying $\lambda=\sigma+\mu-m\leqslant\frac{1}{2}(m^2-m+2)(\mu-1)+m^2-2m$ and $\mu\leqslant m^3(2m-3)$, if the multiplicities of the eigenvalues of $\Gamma$ are integral and both the Krein condition~\cite[Lemma 2.1]{Neumaier} and the absolute bound~\cite[Lemma 2.2]{Neumaier} are satisfied, then we call the parameters $(v,k,\lambda,\mu)$ {\em feasible} for a strongly regular graph.

Let $\Gamma$ be a connected graph. A {\em clique} in $\Gamma$ is a set of pairwise adjacent vertices of $\Gamma$, and  a  {\em coclique} in $\Gamma$ is a set of pairwise non-adjacent vertices of $\Gamma$. The number of vertices in a clique or coclique is called the {\em order} of the clique or coclique. A clique $C$ in $\Gamma$ is called {\em maximal} if there is no clique in $\Gamma$ that contains $C$ and at least one other vertex of $V(\Gamma)\backslash C$.   A {\em complete graph} $K_n$ is a graph whose vertex set is a clique of order $n$. For a vertex $x$ of $\Gamma$, if $\Gamma(x)$ contains a coclique $\bar{C}$ of order $s$, then the subgraph induced on $\bar{C}\cup \{x\}$ is called the {\em $s$-claw}. The {\em adjacency matrix} $A=A(\Gamma)$ of $\Gamma$ is the matrix whose rows and columns are indexed by vertices of $\Gamma$ and the ($x, y$)-entry is $1$ whenever $x$ and $y$ are adjacent and $0$ otherwise. The {\em eigenvalues} of $\Gamma$ are the eigenvalues of $A$.

The following well-known result is called the {\em Interlacing Theorem}, and it shows that for a graph $\Gamma$, the eigenvalues of an induced subgraph of $\Gamma$ interlace the eigenvalues of $\Gamma$. 
\begin{theorem}\label{1}{\rm (Cf.\cite{haem})}
Let $m \leqslant n$ be two positive integers. Let
$A$ be an $n\times n$ matrix, that is similar to a (real) symmetric matrix, and let
$B$ be a principal $m \times m$ submatrix of $A$. Then, $$\theta_{n-m+i}(A)\leqslant \theta_i(B)\leqslant \theta_i(A)$$
holds for $i=1,\ldots , m$, where $A$ has eigenvalues $\theta_1(A) \geqslant \theta_2(A) \geqslant \ldots\geqslant \theta_n(A)$ and B has eigenvalues
$\theta_1(B) \geqslant \theta_2(B) \geqslant \ldots \geqslant \theta_m(B)$.\\
\end{theorem}

\begin{remark}\label{equitable}
A partition $\Pi=\{P_1, P_2, \ldots, P_m\}$ of the vertex set of a graph $\Gamma$ is called {\em equitable} if there exist non-negative integers $q_{ij}$ $(1\leqslant i, j \leqslant m)$ such that each vertex in $P_i$ has exactly $q_{ij}$ neighbors in $P_j$.  
Moreover, an eigenvalue of the {\em quotient matrix} $Q=(q_{ij})$ of the equitable partition $\Pi=\{P_1, P_2, \ldots, P_m\}$ is also an eigenvalue of the graph $\Gamma$~\cite[Lemma 2.3.1]{Spectra}. 
Together with Theorem~\ref{1}, this shows that eigenvalues of the quotient matrix of an equitable partition of the vertex set of an induced subgraph of the graph $\Gamma$ interlace those of $\Gamma$.
\end{remark}

\section{Maximal-clique polynomial}
\label{sec:poly}

In this section, we introduce a cubic polynomial (see Remark~\ref{maxclipoly}) that gives a new bound on the order of a maximal clique in an amply regular graph. Indeed, this cubic polynomial says that a maximal clique in an amply regular graph is either large or small (under certain assumptions). 

For positive integers $a$ and $t$, let us consider the graph with $a+t+1$ vertices consisting of a complete graph $K_{a+t}$ together with a vertex $x$ that is adjacent to precisely $a$ vertices of $K_{a+t}$. We denote this graph by $H(a,t)$.
Note that the vertex partition of $\Gamma = H(a,t)$ with parts $\{x\}$, $\Gamma(x)$ and their complement is equitable with quotient matrix
\[
  Q = \begin{bmatrix}
    0 & a & 0 \\
    1 & a-1 & t \\
    0 & a & t-1
  \end{bmatrix}.
\]

For a graph with smallest eigenvalue $-m$ containing $H(a,t)$ as an induced subgraph, the following lemma gives a relationship between $a, t$ and $m$.
\begin{lemma}\label{lem:hat}
  Let $\Gamma$ be a graph with smallest eigenvalue $-m$ that contains $H(a,t)$ as an induced subgraph.
  Then 
  \begin{equation}
    \label{eqn:atbound}
    (a-m(m-1))(t-(m-1)^2) \leqslant (m(m-1))^2.
  \end{equation}
\end{lemma}
\begin{proof}
Note that Remark~\ref{equitable} says that the smallest eigenvalue of $Q$ is at least $-m$. Thus $\det(mI+Q) \geqslant 0$, from which $(\ref{eqn:atbound})$ follows directly.
\end{proof}

In the following lemma, we show that the parameter $a$ in Lemma~\ref{lem:hat} is small when any two non-adjacent vertices of $\Gamma$ have few common neighbors.

\begin{lemma}\label{lem:aboundmu}
  Let $\Gamma$ be a graph with smallest eigenvalue $-m$ such that any two non-adjacent vertices have at most $\mu$ common neighbors.
  Suppose that  $\Gamma$ has a maximal clique $C$ of order $c\geqslant (m-1)(4m-1)$ and a vertex $y \not \in C$ with $a$ neighbors in $C$.
  If $\mu < \frac{c+m-1 + \sqrt{D}}{2}$ then $ a \leqslant \frac{c+m-1 - \sqrt{D}}{2}$, where $D=(c+m-1)(c-(m-1)(4m-1))$.
\end{lemma}
\begin{proof}
  Note that $\Gamma$ contains $H(a,c-a)$ as an induced subgraph. Then \eqref{eqn:atbound} becomes
  \[
   a^2-a(c+m-1)+m(m-1)(c+m-1) \geqslant 0
  \]
  Since $c\geqslant (m-1)(4m-1)$, we know that $(c+m-1)(c-(m-1)(4m-1))\geqslant 0$. Hence either
  \[
    a \leqslant \frac{c+m-1 - \sqrt{(c+m-1)(c-(m-1)(4m-1))}}{2}
  \]
  or
  \[
    a \geqslant \frac{c+m-1 + \sqrt{(c+m-1)(c-(m-1)(4m-1))}}{2}
  \]
  holds. But, since $C$ is maximal, we must have $a \leqslant \mu$.
\end{proof}

We apply the proof of Lemma~\ref{lem:aboundmu} to strongly regular graphs having a Delsarte clique to obtain the following result.

\begin{proposition}
Let $\Gamma$ be a strongly regular graph with parameters $(v,k,\lambda,\mu)$ having smallest eigenvalue $-m$. Assume that $k\geqslant m^2(4m-5)$ and that $\Gamma$ contains a Delsarte clique. Then either
\begin{center}
$\mu \leqslant  \frac{k+m^2 - \sqrt{(k+m^2)(k-4m^3+5m^2)}}{2}$~ or ~$\mu \geqslant \frac{k+m^2 + \sqrt{(k+m^2)(k-4m^3+5m^2)}}{2}$.
\end{center}
\end{proposition}

\begin{proof}
Let $C$ be a Delsarte clique in $\Gamma$. 
Then $C$ has order $c=1+\frac{k}{m}$.
By \cite[Proposition 1.3.2]{bcn89}, every vertex outside $C$ has $\frac{\mu}{m}$ neighbors in $C$. 
Since $k\geqslant m^2(4m-5)=m(m-1)(4m-1)-m$, we know that $c=1+\frac{k}{m}\geqslant (m-1)(4m-1)$. 
Then from the proof of Lemma~\ref{lem:aboundmu}, we know that either  
\begin{center}
$\frac{\mu}{m} \leqslant \frac{c+m-1 - \sqrt{(c+m-1)(c-(m-1)(4m-1))}}{2}$~  or  ~$\frac{\mu}{m} \geqslant \frac{c+m-1 + \sqrt{(c+m-1)(c-(m-1)(4m-1))}}{2}$.
\end{center}
Replace $c$ by $1+\frac{k}{m}$ to obtain that either
\begin{center}
    $\mu \leqslant \frac{k+m^2 - \sqrt{(k+m^2)(k-4m^3+5m^2)}}{2}$~ or  ~$\mu \geqslant \frac{k+m^2 + \sqrt{(k+m^2)(k-4m^3+5m^2)}}{2}$.
\end{center}
This finishes the proof.
\end{proof}

Next we provide a technical lemma for adjacency for cliques in edge-regular graphs.

\begin{lemma}\label{lem:w}
  Let $\Gamma$ be an edge-regular graph with parameters $(v,k,\lambda)$ having a clique $C$ of order $c$.  For a vertex $x$ in $ C$, if every vertex in $\Gamma(x) \backslash C$ has at most $n$ neighbors in $C \backslash \{x \}$, then 
  \[
    \frac{(c-1)(\lambda-(c-2))}{(k-(c-1))} \leqslant n.
  \]
\end{lemma}
\begin{proof}
 Note that every vertex in $\Gamma(x)$ has $\lambda$ neighbors in $\Gamma(x)$. 
 Then every vertex in $C \backslash \{ x \}$ has $\lambda-(c-2)$ neighbors in $\Gamma(x) \backslash C$. 
 Thus, there are $(c-1)(\lambda-(c-2))$ edges between $C \backslash \{ x \}$ and $\Gamma(x) \backslash C$. 
 Since every vertex in $\Gamma(x) \backslash C$ has at most $n$ neighbors in $C \backslash \{x \}$, we have $(c-1)(\lambda-(c-2)) \leqslant n(k-(c-1))$, as required.
\end{proof}

As an immediate consequence of Lemma~\ref{lem:w}, we have the following corollary.

\begin{corollary}\label{boundonmu}
Let $\Gamma$ be an amply regular graph with parameters $(v,k,\lambda, \mu)$ having a clique $C$ of order $c$. Then 
\[
    \frac{(c-1)(\lambda-(c-2))}{(k-(c-1))} \leqslant \mu-1.
  \]
\end{corollary}
\begin{proof}
Let $x$ be a vertex in $C$ and let $y$ be a vertex in $\Gamma(x)\backslash C$. 
Note that there is a vertex $z$ in $C \backslash \{x\}$ such that $z$ is not adjacent to $y$. 
Then, since $d(y,z)=2$, the vertex $y$ has at most $\mu-1$ neighbors in $\Gamma(x)\backslash C$. 
By Lemma~\ref{lem:w}, we obtain that $\mu-1$ is at least $\frac{(c-1)(\lambda-(c-2))}{(k-(c-1))}$, as required.
\end{proof}

We combine Lemma~\ref{lem:aboundmu} and Lemma~\ref{lem:w} to establish the following lemma.

\begin{lemma}\label{lem:srg1}
  Let $\Gamma$ be an amply regular graph with parameters $(v,k,\lambda,\mu)$ having smallest eigenvalue $-m$.
  Suppose that $\Gamma$ has a maximal clique of order $c\geqslant(m-1)(4m-1)$ such that $\mu < \frac{c+m-1 + \sqrt{D}}{2}$, where $D=(c+m-1)(c-(m-1)(4m-1))$.
  Then 
 \begin{equation} \label{ineq:polybound1}
   (c+m-3)(k-c+1)-2(c-1)(\lambda-c+2)\geqslant(k-c+1)\sqrt{D}.
  \end{equation}
\end{lemma}

\begin{proof}
 Let $x$ be a vertex in $C$. For a vertex $y$ in $\Gamma(x) \backslash C$, we denote the number of neighbors of $y$ in $C\backslash\{x\}$ by $n_y$. Let $n={\rm max}\{n_y ~|~ y\in \Gamma(x) \backslash C\}$. 
 Then by  Lemma~\ref{lem:w}, we have  $\frac{(c-1)(\lambda-(c-2))}{(k-(c-1))} \leqslant n$. 
 Let $z$ be a vertex in $\Gamma(x)\backslash C$ having $n$ neighbors in $C\backslash\{x\}$, that is, the vertex $z$ has $n+1$ neighbors in $C$ including $x$. 
 By Lemma~\ref{lem:aboundmu}, we have $n+1\leqslant \frac{c+m-1 - \sqrt{D}}{2}$. 
 Thus, we obtain that 
$$\frac{(c-1)(\lambda-(c-2))}{(k-(c-1))}+1 \leqslant n+1\leqslant \frac{c+m-1 - \sqrt{D}}{2},$$ from which $(\ref{ineq:polybound1})$ follows directly.
\end{proof}

Alternatively, Lemma~\ref{lem:srg1} can be written as the following lemma.

\begin{lemma}\label{lem:srg2}
  Let $\Gamma$ be an amply regular graph with parameters $(v,k,\lambda,\mu)$ having smallest eigenvalue $-m$ such that $\mu > m(m-1)$.
  Suppose that $\Gamma$ has a maximal clique $C$ of order $c > \frac{ \mu^2}{\mu-m(m-1)}-m+1$.
  Then 
 \begin{equation}
    \label{ineq:polybound2}
    ((c+m-3)(k-c+1)-2(c-1)(\lambda-c+2))^2-(k-c+1)^2(c+m-1)(c-(m-1)(4m-1)) \geqslant 0.
  \end{equation}
\end{lemma}

\begin{proof}

Let $D=(c+m-1)(c-(m-1)(4m-1))$.  
Since $\mu > m(m-1)$, the inequality $c > \frac{ \mu^2}{\mu-m(m-1)}-m+1$ implies that $(\mu-m(m-1))c>\mu^2-(m-1)(\mu-m(m-1))$, which is equivalent to $(2\mu-c-m+1)^2<D$. 
Hence $2\mu-c-m+1\leqslant|2\mu-c-m+1|<\sqrt{D}$, that is, $\mu < \frac{c+m-1 + \sqrt{D}}{2}$.
Since $(\mu-2m(m-1))^2\geqslant 0$, we have that $\mu^2\geqslant 4m(m-1)(\mu-m(m-1))$. 
Then $\mu>m(m-1)$ implies that $\frac{\mu^2}{\mu-m(m-1)}\geqslant 4m(m-1)$, and this shows that $c>\frac{\mu^2}{\mu-m(m-1)}-m+1 \geqslant (m-1)(4m-1)$.  
Now we can apply Lemma~\ref{lem:srg1}, to obtain the inequality $(\ref{ineq:polybound1})$, and this implies the inequality $(\ref{ineq:polybound2})$, as required.
\end{proof}

\begin{remark}\label{maxclipoly}
We denote the polynomial on the left hand side of the inequality \eqref{ineq:polybound2}  by $M_\Gamma(c)$, and we will call it the {\em maximal-clique polynomial}. Note that $M_\Gamma(c)$ is a cubic polynomial in the variable $c$ and that the leading coefficient of $M_\Gamma(c)$ is positive.
\end{remark}


\section{The claw-bound and cliques}
\label{sec:claw}

In this section, we recall the {\em claw-bound} which was found by several authors \cite[Section 3]{KIPR}. 
From the claw-bound, we will show that if an amply regular graph  $\Gamma$ with parameters $(v,k,\lambda,\mu)$ does not contain a coclique of order $\bar{c}$ in a local graph (a graph induced on the set of neighbors of a vertex), then $\Gamma$ contains a clique of order at least $2+\lambda-(\bar{c}-2)(\mu-1)$.

The claw-bound is given below, and it follows from the principle of inclusion and exclusion (\cite{Gav,KP}).
\begin{lemma}\label{clawbound}{\rm (Cf.\cite{Gav,KP})}
Let $\Gamma$ be an amply regular graph with parameters $(v,k,\lambda,\mu)$. Let $x$ be a vertex of $\Gamma$ and let $\bar{C}$ be a coclique of order $\bar{c}\geqslant2$ in $\Gamma(x)$. Then 

$${\bar{c} \choose 2}(\mu-1) \geqslant \bar{c}(\lambda+1)-k$$
\end{lemma}

\begin{remark}\label{clawclique}
Lemma~\ref{clawbound} says that if ${\bar{c} \choose 2}(\mu-1)<\bar{c}(\lambda+1)-k$, then the graph induced on $\Gamma(x)$ contains no coclique of order $\bar{c}$, i.e., the graph $\Gamma$ does not contain a $\bar{c}$-claw.
\end{remark}

From the claw-bound, we give a bound on the order of a clique in an amply regular graph when the graph does not contain a $\bar{c}$-claw.

\begin{lemma}\label{largeclique1}
Let $\Gamma$ be an amply regular graph with parameters $(v,k,\lambda,\mu)$. If ${\bar{c} \choose 2}(\mu-1) < \bar{c}(\lambda+1)-k$ for some integer $\bar{c}\geqslant 2$, then $\Gamma$ contains a clique of order at least $2+\lambda-(\bar{c}-2)(\mu-1)$.
\end{lemma}
\begin{proof}
Let $x$ be a vertex of $\Gamma$ and let $s$ be the maximum number such that the graph induced by $\Gamma(x)$ contains a coclique of order $s$. 
Note that, by Remark~\ref{clawclique}, we know that $s\leqslant \bar{c}-1$. 
Assume that $\{y_1, y_2, \ldots, y_s\}$ is a coclique of order $s$ in $\Gamma(x)$. 
The vertex $y_1$ has $\lambda$ neighbors in $\Gamma(x)\backslash\{y_2, \dots, y_s\}$ and that for each $i \in \{2,\dots,s\}$, the vertices $y_1$ and $y_i$ have at most $\mu-1$ common neighbors in $\Gamma(x)\backslash\{y_2, \dots, y_s\}$. 
Hence, there are at least $\lambda-(s-1)(\mu-1)$ vertices in $\Gamma(x)$ that are adjacent to $y_1$ but not adjacent to $y_i$ for all $i \in \{2,\dots,s\}$. 
If two such vertices were not adjacent, then those two vertices together with $y_2, \ldots, y_s$ would induce a coclique of order $s+1$, a contradiction. 
Thus, the graph induced by $\Gamma(x)$ contains a clique of order at least $1+\lambda-(s-1)(\mu-1)$ and hence, $\Gamma$ contains a clique of order at least $2+\lambda-(s-1)(\mu-1)\geqslant 2+\lambda-(\bar{c}-2)(\mu-1)$. 
\end{proof}



As an application of the maximal clique polynomial, we prove that there cannot exist any strongly regular graph having the feasible parameters $(v,k,\lambda,\mu)$ in the following theorem.

\begin{theorem}\label{nosrg}
Let $m\geqslant4$ be an integer. Then, there are no strongly regular graphs with the following parameters:
\begin{itemize}
\item[] $v=1+k+k(k-\lambda-1)/\mu$,
\item[] $k=(m+1)(m(2-\mu)+2\lambda)/2+1$,
\item[] $\lambda=\frac{(m-3)^5+15(m-3)^4+91(m-3)^3+283(m-3)^2)}{2}+226(m-3)+148$,
\item[] $\mu=(m-3)^3+10(m-3)^2+33(m-3)+38$.
\end{itemize}
\end{theorem}

\begin{proof}
Suppose that there exists a strongly regular graph $\Gamma$ with such parameters $(v,k,\lambda,\mu)$ for some integer $m\geqslant 4$. Set $\bar{c}=m+2$. 
Then ${\bar{c} \choose 2}(\mu-1)<\bar{c}(\lambda+1)-k$, and, by Lemma~\ref{largeclique1}, we know that $\Gamma$ contains a clique of order at least $2+\lambda-m(\mu-1)$. 
Thus, the graph $\Gamma$ contains a maximal clique $C$ of order $c_1\geqslant 2+\lambda-m(\mu-1)$.  
Note that the Delsarte bound implies that $2+\lambda-m(\mu-1)\leqslant c_1\leqslant 1+\frac{k}{m}$. 

Recall the maximal-clique polynomial $M_\Gamma(c)$. 
Since $m\geqslant 4$, we know that 
$$c_1\geqslant 2+\lambda-m(\mu-1)> \frac{\mu^2}{\mu-m(m-1)}-m+1.$$ 
Then Lemma~\ref{lem:srg2} implies that $M_\Gamma(c_1)\geqslant 0$.  
Note that $M_\Gamma(0)>0$, $M_\Gamma(2+\lambda-m(\mu-1))<0$ and $M_\Gamma(1+\frac{k}{m})<0$. 
But this is not possible since $M_\Gamma(c)$ is a cubic polynomial with a positive leading coefficient (Remark~\ref{maxclipoly}). 
Therefore, there are no strongly regular graphs with such parameters $(v,k,\lambda,\mu)$ for all integer $m\geqslant 4$.
\end{proof}

\section{Acknowledgements}

Gary Greaves is partially supported by the Singapore Ministry of Education Academic Research Fund (Tier 1); grant numbers: RG29/18 and RG21/20.

Jack H. Koolen is partially supported by the National Natural Science Foundation of China (No.12071454), Anhui Initiative in Quantum Information Technologies (No. AHY150000) and by the project ``Analysis and Geometry on Bundles" of Ministry of Science and Technology of the People’s Republic of China.

Jongyook Park is partially supported by Basic Science Research Program through the National Research Foundation of Korea funded by Ministry of Education (NRF-2017R1D1A1B03032016) and the National Research Foundation of Korea (NRF) grant funded by the Korea government (MSIT) (NRF-2020R1A2C1A01101838).

\section{Appendix}

In this appendix, we give tables of parameters $(v,k,\lambda,\mu)$ for nonexistent strongly regular graphs with smallest eigenvalue $-4,-5,-6$ or $-7$. Note that all of those parameters $(v,k,\lambda,\mu)$ are feasible (see the definition in Section~\ref{sec:2}). 
In the tables below, `forbidden range' means that the graph does not contain a maximal clique of order in the forbidden range (Lemma~\ref{lem:srg2}), `Delsarte bound' means $\lfloor 1+\frac{k}{m}\rfloor$, where $-m$ is the smallest eigenvalue, and `guaranteed clique order' means the graph contains a (maximal) clique of order at least that guaranteed clique order (Lemma~\ref{largeclique1}). 

\begin{table}[htbp]
  \begin{center}
  \begin{tabular}{|c | c  c c| c|c | c|}
\hline
    $\mu$ & $v$ & $k$ & $\lambda$ & forbidden range & Delsarte bound & guaranteed clique order \\
    \hline
    58 & 23276 & 1330 & 372 &  [71, 340] & 333 & 146 \\
    62 & 25025 & 1426 & 399 &  [74, 368] & 357 & 157 \\
    80 & 27455 & 1696 & 480 &  [92, 450] & 425 & 166 \\
    82 & 38875 & 2046 & 569 &  [94, 539] & 512 & 247 \\
\hline
  \end{tabular}
  \end{center}
\caption{Parameters for nonexistent strongly regular graphs with smallest eigenvalue $-4$}
\end{table}

\begin{table}[htbp]
  \begin{center}
  \begin{tabular}{|c | c  c c| c|c | c|}
\hline
    $\mu$ & $v$ & $k$ & $\lambda$ & forbidden range & Delsarte bound & guaranteed clique order \\
    \hline
  115 & 133570 & 4365 & 960 & [136, 885] & 874 & 278\\
  122 & 230958 & 5917 & 1276 &  [142, 1202] & 1184 & 673\\
  150 & 235586 & 6625 & 1440 & [170, 1367] & 1326 & 697\\
  152 & 317628 & 7747 & 1666 & [172, 1593] & 1550 & 913\\
  168 & 328560 & 8283 & 1786 & [187, 1714] & 1657 & 953\\
  170 & 259000 & 7395 & 1610 & [189, 1538] & 1480 & 767\\
  172 & 309016 & 8127 & 1758 & [191, 1686] & 1626 & 905\\
  205 & 225885 & 7580 & 1675 & [224, 1605] & 1517 & 453\\
  214 & 404587 & 10374 & 2241 & [233, 2170] & 2075 & 1178\\
  240 & 314116 & 9675 & 2122 & [258, 2052] &1936 & 929\\
  240 & 485815 & 12040 & 2595 &  [258, 2524] & 2409 & 1402\\
\hline
  \end{tabular}
  \end{center}
\caption{Parameters for nonexistent strongly regular graphs with smallest eigenvalue $-5$}
\end{table}

\begin{table}[htbp]
  \begin{center}
  \begin{tabular}{|c | c  c c| c|c | c|}
\hline
    $\mu$ & $v$ & $k$ & $\lambda$ & forbidden range & Delsarte bound & guaranteed clique order \\
    \hline
 201 & 545832 & 11451 & 2070 &  [232, 1926] & 1909 & 472\\
 204 & 895665 & 14784 & 2628 &  [235, 2484] & 2465  &1412\\
 206 & 1331968 & 18122 & 3186 & [237, 3042] & 3021 & 1958\\
210 & 997920 & 15834 & 2808 & [241, 2664] & 2640 & 1556\\
 212 & 1371657 & 18656 & 3280 & [242, 3137] & 3110 & 2016\\
 252 & 1352572 & 20196 & 3570 & [282, 3428] & 3367 & 2066\\
 254 & 1756209 & 23108 & 4057 & [284, 3915] & 3852 & 2541\\
 264 & 717574 & 15048 & 2722 &  [293, 2582] & 2509 & 620\\
 267 & 886222 & 16821 & 3020 & [296, 2880] & 2804 & 1160\\
 270 & 1112320 & 18954 & 3378 & [299, 3237] & 3160 & 1766\\
 273 & 1423818 & 21567 & 3816 & [302, 3675] & 3595 & 2186\\
 276 & 1867591 & 24840 & 4364 & [305, 4223] & 4141 & 2716\\
 280 & 1026875 & 18544 & 3318 & [309, 3178] & 3091 & 1646\\
 315 & 855570  & 17949 & 3248 & [344, 3110] & 2992  &738\\
 324 & 1462209 & 23808 & 4232 &[353, 4093]& 3969 &2296\\
 327 & 1791882 & 26481 & 4680& [356, 4540] &4414 &2726\\
 330 & 2232000 & 29694 & 5218& [359, 5078] &4950 &3246\\
 380 & 1503625 & 26144 & 4668 &[408, 4530] &4358 &2396\\
 390 & 2223180 & 32214 &5688 &[418, 5549] &5370 &3356\\
 438 & 1148448 & 24522 &4446 &[466, 4311] &4088 &952\\
 440 & 1212001 & 25250 &4569 &[468, 4433] &4209 &1498\\
 450 & 1605240 & 29394 &5268 &[478, 5132] &4900 &2576\\
 456 & 1920621 & 32370 &5769 &[484, 5632] &5396 &3041\\
 468 & 2835028 & 39852 &7026 &[496, 6888] &6643 &4226\\
 470 & 3039520 & 41354 &7278 &[498, 7140] &6893 &4466\\
 472 & 3263897 & 42946 &7545 &[500, 7407] &7158 &4721\\
\hline
  \end{tabular}
  \end{center}
\caption{Parameters for nonexistent strongly regular graphs with smallest eigenvalue $-6$}
\end{table}

\begin{table}[htbp]
  \begin{center}
  \begin{tabular}{|c | c  c c| c|c | c|}
\hline
    $\mu$ & $v$ & $k$ & $\lambda$ & forbidden range & Delsarte bound & guaranteed clique order \\
    \hline
 322& 1769600 &25753 &3948 &[365, 3702] &3680 &1061\\
 329 &4271421 &40460 &6055 &[372, 5810] &5781 &3761\\
 338 &6057152 &48841 &7260 &[380, 7016] &6978 &4903\\
 364 &2747626 &34125 &5180 &[406, 4937] &4876 &2278\\
 369 &5619713 &49152 &7331 &[411, 7088] &7022 &4757\\
 372 &5783316 &50065 &7464 &[414, 7221] &7153 &4869\\
 382 &2369800 &32463 &4958 &[424, 4716] &4638 &1531\\
 392 &5873750 &51793 &7728 &[434, 7485] &7400 &4993\\
 394 &7404736 &58305 &8660 &[436, 8417] &8330 &5911\\
 417 &7593750 &60743 &9028 &[458, 8786] &8678 &6118\\
 467 &2897225 &39688 &6063 &[508, 5824] &5670 &1871\\
 474 &4178176 &48025 &7260 &[515, 7021] &6861 &3951\\
 483 &7331625 &64232 &9583 &[524, 9342] &9177 &6211\\
 486 &9133968 &71921 &10684 &[526, 10443] &10275 &7291\\
 522 &3891200 &48633 &7388 &[562, 7150] &6948 &3222\\
 522 &7314000 &66693 &9968 &[562, 9728] &9528 &6323\\
 532 &4231150 &51198 &7763 &[572, 7525] &7315 &3517\\
 539 &7818591 &70070 &10465 &[579, 10226] &10011 &6701\\
 595 &3189151 &46998 &7217 &[635, 6982] &6715 &1279\\
 630 &3325728 &49385 &7588 &[670, 7354] &7056 &1300\\
 630 &10072881 &85988 &12817 &[670, 12578] &12285 &8416\\
 634 &11915776 &93825 &13940 &[673, 13701] &13404 &9511\\
 665 &12031999 &96558 &14357 &[704, 14118] &13795 &9711\\
 689 &5860526 &68575 &10380 &[728, 10144] &9797 &5566\\
 714 &3844176 &56525 &8680 &[753, 8447] &8076 &1552\\
 735 &5498361 &68600 &10423 &[774, 10188] &9801 &4553\\
 742 &6244525 &73458 &11123 &[781, 10888] &10495 &5197\\
 746 &6729536 &76465 &11556 &[785, 11321] &10924 &6343\\
 762 &9236916 &90551 &13582 &[801, 13346] &12936 &8257\\
 763 &9431401 &91560 &13727 &[802, 13490] &13081 &8395\\
 770 &4190144 &61285 &9408 &[809, 9176] &8756 &1720\\
 777 &12822369 &107744 &16051 &[816, 15813] &15393 &10621\\
 780 &13752261 &111800 &16633 &[819, 16395] &15972 &11182\\
 816 &12114648 &107321 &16024 &[855, 15787] &15332 &10321\\
 833 &4660685 &67228 &10311 &[872, 10079] &9605 &2825\\
 841 &5217895 &71478 &10925 &[880, 10692] &10212 &3367\\
 849 &5861241 &76120 &11595 &[888, 11362] &10875 &4813\\
 882 &9888000 &100793 &15148 &[921, 14913] &14400 &8983\\
 888 &10974960 &106553 &15976 &[927, 15740] &15222 &9769\\
 889 &11171083 &107562 &16121 &[928, 15885] &15367 &9907\\
 900 &5357638 &74925 &11468 &[939, 11236] &10704 &3379\\
 903 &14473410 &123403 &18396 &[942, 18159] &17630 &12084\\
 918 &5105376 &73865 &11332 &[957, 11101] &10553 &3081\\
 924 &6973876 &86625 &13160 &[963, 12927] &12376 &5778\\
 980 &13835098 &125685 &18788 &[1018, 18552] &17956 &11937\\
 990 &5550960 &79985 &12268 &[1028, 12037] &11427 &3369\\
 1007 &8792525 &101548 &15363 &[1045, 15130] &14507 &8323\\
\hline
  \end{tabular}
  \end{center}
\caption{Parameters for nonexistent strongly regular graphs with smallest eigenvalue $-7$}
\end{table}

\end{document}